\documentclass{article} 

\usepackage{amsmath}
\usepackage{amsthm}
\usepackage{amssymb}

\usepackage{url}
\usepackage{xcolor}

\usepackage[colorlinks=true,urlcolor=blue,linkcolor=blue,citecolor=blue,plainpages=false,pdfpagelabels]{hyperref}

\newtheorem{theorem}{Theorem}[section]
\newtheorem*{theorem*}{Theorem}

\newtheorem{proposition}[theorem]{Proposition}
\newtheorem{corollary}[theorem]{Corollary}
\newtheorem{lemma}[theorem]{Lemma}
\newtheorem{remark}[theorem]{Remark}

\newcommand{\bprop}{\begin{proposition}}
\newcommand{\eprop}{\end{proposition}}

\newcommand{\blem}{\begin{lemma}}
\newcommand{\elem}{\end{lemma}}

\newcommand{\N}{{\mathbb N}}

\newcommand{\eps}{\varepsilon}

\newcommand{\roneq}{\sigma_1}
\newcommand{\rtwoq}{\sigma_2}
\newcommand{\rthreeq}{\sigma_3}
\newcommand{\rfourq}{\sigma_4}
\newcommand{\rfiveq}{\Lambda}

\newcommand{\rarxn}{\Gamma_1}
\newcommand{\raryn}{\Gamma_2}
\newcommand{\uraryn}{\Gamma_3}
\newcommand{\traryn}{\Gamma_4}
\newcommand{\urarxn}{\Gamma_5}
\newcommand{\trarxn}{\Gamma_6}
\newcommand{\rxnyn}{\Omega}
\newcommand{\uraryntilde}{\widetilde{\uraryn}}

\newcommand{\uraryncat}{\Gamma_0}
\newcommand{\pcat}{P_0}
\newcommand{\rarxns}{\Sigma_1}
\newcommand{\raryns}{\Sigma_2}
\newcommand{\uraryns}{\Sigma_3}
\newcommand{\traryns}{\Sigma_4}
\newcommand{\urarxns}{\Sigma_5}

\newcommand{\rxnyns}{\Theta}

\title{Rates of asymptotic regularity for the alternating Halpern-Mann iteration}

\author{Lauren{\c t}iu Leu{\c s}tean${}^{a,c,d}$ and Pedro Pinto${}^{b}$\\[2mm]
	\footnotesize ${}^{a}$LOS, Faculty of Mathematics and Computer Science, University of Bucharest\\
	\footnotesize ${}^{b}$Department of Mathematics, Technische Universit{\"a}t Darmstadt\\ 
	\footnotesize ${}^{c}$ Simion Stoilow Institute of Mathematics of the Romanian Academy\\
	\footnotesize ${}^{d}$ Institute for Logic and Data Science, Bucharest\\[1mm]
	\footnotesize Emails: \protect\url{laurentiu.leustean@unibuc.ro}, \protect\url{pinto@mathematik.tu-darmstadt.de}
}

\begin{document}

\date{}

\maketitle

\begin{abstract}
In this paper we extend to  $UCW$-hyperbolic spaces the quantitative asymptotic regularity results for the alternating Halpern-Mann iteration obtained by Dinis and the second author  for CAT(0) spaces.  These results are new even for uniformly convex normed spaces. Furthermore, for a particular choice of the parameter sequences, we compute linear rates of asymptotic regularity in $W$-hyperbolic spaces and quadratic rates of $T$- and $U$-asymptotic regularity in CAT(0) spaces.\\



\end{abstract}

\section{Introduction}\label{intro}

Let  $X$ be a CAT(0) space, $C\subseteq X$   a convex subset, and  $T, U:C\to C$ two nonexpansive mappings.
Dinis and the second author introduced recently \cite{DinPin21} the
alternating Halpern-Mann iteration  $(x_n)$ as follows: if $x_0,u \in C$ and $(\alpha_n)_{n\in\N}$, $(\beta_n)_{n\in\N}$ 
are sequences in $[0,1]$, then 
\begin{equation}\label{e:MannHalpern}
\begin{cases}
x_{2n+1}&=(1-\alpha_n)Tx_{2n} + \alpha_n u,\\
x_{2n+2}&=(1-\beta_n)Ux_{2n+1} + \beta_nx_{2n+1}.
\end{cases}
\end{equation}
Thus, the alternating Halpern-Mann iteration is an iterative scheme that alternates between the well-known 
Halpern \cite{Hal67} and  Mann \cite{Man53,Kra55,Gro72}  iterations. As pointed out in \cite{DinPin21}, 
$(x_n)$ is a generalization of the Tikhonov-Mann iteration, introduced by Cheval and the first author \cite{CheLeu22} as an extension of a modified Mann iteration studied by Yao, Zhou, and Liou \cite{YaoZhoLio09} and rediscovered  by Bo{\c t}, Csetnek, and Meier \cite{BotCseMei19}. The Tikhonov-Mann iteration was recently shown by Cheval, Kohlenbach, and the first author \cite{CheKohLeu23} 
to be essentially the modified Halpern iteration, a generalization of the Halpern iteration due to Kim and Xu \cite{KimXu05}.

Quantitative theorems, providing rates of asymptotic regularity and rates of metastability for $(x_n)$, were proved in \cite{DinPin21}
by using techniques developed in \cite{FerLeuPin19}, under the assumption  that  $T$ and $U$  have common fixed points and  
the parameter sequences $(\alpha_n)$,  $(\beta_n)$ satisfy certain hypotheses. In particular, the results in \cite{DinPin21} show that, 
in CAT(0) spaces, the alternating Halpern-Mann iteration converges strongly to a common fixed point of the mappings.

Rates of asymptotic regularity were obtained  for the Halpern iteration by Kohlenbach and the first author 
\cite{KohLeu12a} in the much more general setting of $W$-hyperbolic spaces \cite{Koh05}, and for the 
Mann iteration by the first author \cite{Leu07} for $UCW$-hyperbolic spaces \cite{Leu07,Leu10}, a class of geodesic 
spaces that generalize both CAT(0) spaces and uniformly convex normed spaces. A natural question is then 

\begin{center}
\textit{to extend to $UCW$-hyperbolic spaces the quantitative asymptotic regularity results obtained in \cite{DinPin21} for 
CAT(0) spaces}.
\end{center}
In this paper,  we give an answer to this question  by computing  uniform rates of ($T$- and $U$-)asymptotic regularity 
for the alternating Halpern-Mann iteration.  These results are new even for uniformly convex normed spaces.  

Moreover, we  compute for the first time, for a particular choice of the 
parameter sequences, linear rates of asymptotic regularity in $W$-hyperbolic spaces and quadratic rates of 
$T$- and $U$-asymptotic regularity in CAT(0) spaces. In order to obtain the linear rates, we adapt to our setting 
a lemma on sequences of real numbers due to Sabach and Shtern \cite{SabSht17}.

Let us define the asymptotic regularity notions that we use. Assume that $(X,d)$ is a metric space. 
A sequence $(z_n)_{n\in\N}$ in $X$ is asymptotically regular if $\lim\limits_{n\to\infty} d(z_n,z_{n+1})=0$. 
A rate of asymptotic regularity of $(z_n)$ is a rate of convergence of  $(d(z_n,z_{n+1}))$ towards $0$. Assume, furthermore, 
that  $\emptyset \ne C\subseteq X$, $S:C\to C$ is  a mapping and $(z_n)$ is a sequence in $C$. 
We say that $(z_n)$ is $S$-asymptotically regular if $\lim\limits_{n\to\infty} d(z_n,Sz_n)=0$. A rate of 
$S$-asymptotic regularity of $(z_n)$ is a rate of convergence 
of  $(d(z_n,Sz_n ))$ towards $0$.

We  recall in the following the main quantitative notions that appear throughout the paper.  If $(a_n)_{n\in\N}$ is a 
sequence in   a metric space $(X,d)$, $a\in X$ and $\varphi:\N\to\N$, then 
\begin{enumerate}
\item $\varphi$ is a  rate of convergence of $(a_n)$ (towards $a$) if $\lim\limits_{n\to\infty} a_n =a$ and 
\[\forall k\in\N\, \forall n\geq \varphi(k) \left(d(a_n,a)\leq \frac1{k+1}\right).\]
\item $\varphi$ is a Cauchy modulus of  $(a_n)$ if $(a_n)$ is Cauchy and 
\[\forall k\in\N\, \forall n\geq \varphi(k)\, \forall p\in\N \left(d(a_{n+p},a_n)\leq 
\frac1{k+1}\right).\]
\end{enumerate}
Consider the series $\sum\limits_{n=0}^\infty b_n$, where $(b_n)_{n\in\N}$ is a sequence of nonnegative reals. We say that the series diverges with 
rate of divergence $\theta:\N\to\N$  if $\sum\limits_{i=0}^{\theta(n)} b_i \geq n$ for all $n\in\N$. Furthermore, if $\sum\limits_{n=0}^\infty b_n$  converges, then a Cauchy modulus of the 
series is a Cauchy modulus of the sequence $\left(\sum\limits_{i=0}^{n} b_i\right)$ of partial sums.

\section{Preliminaries}\label{pre}

\subsection{$UCW$-hyperbolic spaces}

Following \cite{CheLeu22}, a metric space $(X,d)$  endowed with 
a  function $W:X\times X\times [0,1]\to X$ is said to be a $W$-space. As $W$ is thought as a generalization to metric spaces of the usual convex 
combination mapping in normed spaces, we use the notation $(1-\lambda)x + \lambda y$ 
for $W(x,y,\lambda)$. Furthermore, we denote a $W$-space simply by $X$.
A nonempty subset $C\subseteq X$ is said to be convex if  for all  $x,y \in C$ and all $\lambda \in [0,1]$,
$(1-\lambda)x + \lambda y \in C$.

Motivated by proof-theoretical considerations, Kohlenbach  introduced in \cite{Koh05} the class of $W$-hyperbolic spaces (called by him `hyperbolic spaces'). 
This class of $W$-spaces is now broadly used in the study of nonlinear iterations. Let us recall that  
a $W$-hyperbolic space is a $W$-space $X$ satisfying, for all $x,y,w,z\in X$ and all $\lambda,~\tilde{\lambda} \in [0,1]$,
$$
\begin{array}{ll}
\text{(W1)} & d(z,(1-\lambda)x + \lambda y) \le  (1-\lambda)d(z,x)+\lambda d(z,y),\\[1mm]
\text{(W2)} & d((1-\lambda)x + \lambda y,(1-\tilde{\lambda})x + \tilde{\lambda} y) =  \vert \lambda-\tilde{\lambda} \vert d(x,y), \\[1mm]
\text{(W3)} & (1-\lambda)x + \lambda y = \lambda y + (1-\lambda)x, \\[1mm]
\text{(W4)} & d((1-\lambda)x + \lambda z,(1-\lambda)y + \lambda w) \le (1-\lambda)d(x,y)+\lambda d(z,w).
\end{array}
$$
The following properties of a $W$-hyperbolic space $X$ are easy to verify:   for all $x,y\in X$ and all  $\lambda\in[0,1]$,\\[2mm]
$
\begin{array}{ll}
\text{(W5)} & (1-\lambda)x + \lambda x =x, \\[1mm]
\text{(W6)} &  1x+0y = 0y+1x = x, \\[1mm]
\text{(W7)} & d(x,(1-\lambda)x + \lambda y)\!=\!\lambda d(x,y) \text{~and~} d(y,(1-\lambda)x + \lambda y)\!=\!(1-\lambda)d(x,y).
\end{array}
$\\[1mm]
An immediate example of a $W$-hyperbolic space is (a convex subset of) a normed space. An extensive 
discussion on the relations between $W$-hyperbolic spaces and other geodesic spaces can be found in \cite[pp. 384-387]{Koh08}.

The class of uniformly convex $W$-hyperbolic spaces was introduced by the first author in \cite{Leu07}, inspired by \cite[p. 105]{GoeRei84}, as 
a generalization to the nonlinear setting of the well-known uniformly convex normed spaces. Thus, 
a uniformly convex $W$-hyperbolic space is a structure $(X,\eta)$, where $X$ is a $W$-hyperbolic  space and
$\eta: (0,\infty)\times (0,2] \to (0,1]$ is a so-called modulus of uniform convexity, satisfying the following:  
for all $r>0$ and $\eps \in (0,2]$, and for any $x,y,a\in X$,
\[\text{if~} d(x,a),  d(y,a) \leq r,  \text{~and~} d(x,y) \geq \eps r, \text{~then~} d\left(\frac{1}{2}x +\frac{1}{2}y, a\right) \leq (1-\eta(r, \eps))r.\]
Following \cite{Leu10}, an $UCW$-hyperbolic space  is a uniformly convex $W$-hyperbolic 
space $(X,\eta)$ with the property that the modulus of uniform convexity is monotone in the following sense:
\begin{center}
for any $\eps \in (0,2]$, if $0<r\leq s$, then $ \eta(s, \eps) \leq \eta(r, \eps)$. 
\end{center}

The class of $UCW$-hyperbolic spaces turns out to be an appropriate setting for obtaining quantitative results on the asymptotic behaviour  of the Mann iteration for 
(asymptotically) nonexpansive mappings (see \cite{Leu07,Leu10,KohLeu10}), as well as of the Picard iteration for firmly nonexpansive mappings 
(see \cite{AriLeuLop14}).

The following property of  $UCW$-hyperbolic spaces will be used in Section \ref{main}. We refer to \cite[Lemma~2.1(iv)]{Leu10} for its proof.

\begin{lemma}\label{prop_eta-s-r} 
Let $(X,\eta)$ be a $UCW$-hyperbolic space. Assume that $r>0$, $\eps \in (0,2]$, and $x,y,a\in X$ satisfy
\[
d(x,a) \leq r,\, d(y,a) \leq r, \text{~and~}\, d(x,y) \geq \eps r.
\]
For all $\lambda \in[0,1]$ and any   $s\geq r$,
\[
d((1-\lambda)x  + \lambda y, a) \leq \left(1-2\lambda(1-\lambda)\eta(s,\eps)\right)r.
\]
\end{lemma}

Obviously, uniformly convex normed spaces are  $UCW$-hyperbolic spaces. CAT(0) spaces \cite{AleKapPet19,BriHae99} are a very important class of spaces in geodesic geometry and geometric group theory.  
The first author showed in \cite{Leu07} that CAT(0) spaces are  $UCW$-hyperbolic spaces with 
the following modulus of uniform convexity $\displaystyle \eta(r,\eps)=\frac{\eps^2}{8}$, 
 that does not depend at all on $r$ and is quadratic in $\eps$.  

\subsection{Alternating Halpern-Mann iteration}

We give, in the very general setting of $W$-spaces, a reformulation of the iterative scheme  
\eqref{e:MannHalpern} introduced in \cite{DinPin21}. Let $X$ be a $W$-space, $C\subseteq X$ a 
convex subset and $T, U:C\to C$ be nonexpansive mappings. 
The alternating Halpern-Mann iteration is defined by the following scheme:
\begin{equation}\label{HM}
\begin{split}
x_{n+1} & = (1-\beta_n)Uy_n + \beta_n y_n,\\
y_{n} & = (1-\alpha_n)T x_n + \alpha_n u,
\end{split}
\end{equation}
where $x_0,u \in C$ and $(\alpha_n)_{n\in\N}$, $(\beta_n)_{n\in\N}$ are sequences in $[0,1]$.

Notice that for the iteration \eqref{e:MannHalpern}, the even terms correspond to our $(x_n)$ and the odd terms to our $(y_n)$. Indeed, this reformulation is just for convenience of notation.

\begin{remark}\label{remark-HM-TM-mH}
\begin{enumerate}
\item If $U=Id_C$, then $(y_n)$ is the Halpern iteration. 
\item\label{TM-mH} If $T=Id_C$, then $(x_n)$ is the Tikhonov-Mann iteration \cite{CheLeu22} and 
$(y_n)$ is the modified Halpern iteration \cite{KimXu05}. 
\item If $T=Id_C$ and $\alpha_n=0$, then $(x_n)$ is the Mann iteration.
\end{enumerate}
\end{remark}
\begin{proof}
\begin{enumerate}
\item By (W5), we have that $x_{n+1}=y_n$.
\item Consider the definitions from  \cite[p.4]{CheKohLeu23} and use \cite[Proposition~3.2]{CheKohLeu23}.
\item Apply (W6) to get that $y_n=x_n$. \qedhere
\end{enumerate}
\end{proof}

Let us recall some useful properties of the alternating Halpern-Mann iteration. In the sequel, $X$ is a $W$-hyperbolic space.

\blem
The following hold for all  $n\in \N$:
\begin{align}
d(y_{n+1}, y_n) & \le (1-\alpha_{n+1}) d(x_{n+1}, x_n) + \vert \alpha_{n+1}-\alpha_n \vert d(Tx_n,u), \label{dynp1-yn}\\
d(x_{n+2}, x_{n+1}) & \le  d(y_{n+1},y_n) + \vert \beta_{n+1} - \beta_n \vert d(Uy_n,y_n). \label{dxnp1-xn}
\end{align}
\elem
\begin{proof}
By the proof of \cite[Lemma~3.3]{DinPin21}. However, as the notations from \cite{DinPin21} are changed in this paper, we give the proofs below.
\begin{align*}
d(y_{n+1}, y_n) & = d((1-\alpha_{n+1})T x_{n+1} + \alpha_{n+1} u, (1-\alpha_n)T x_n + \alpha_n u)) \\
& \leq d((1-\alpha_{n+1})T x_{n+1} + \alpha_{n+1} u, (1-\alpha_{n+1})T x_n+\alpha_{n+1} u)\\
& \quad + d((1-\alpha_{n+1})T x_n + \alpha_{n+1} u, (1-\alpha_n)T x_n + \alpha_n u)\\
&  \leq (1-\alpha_{n+1}) d(T x_{n+1}, Tx_n) + \vert \alpha_{n+1}-\alpha_n \vert d(Tx_n,u)\\ 
& \text{by (W4) and (W2)}\\
&  \leq (1-\alpha_{n+1}) d(x_{n+1}, x_n) + \vert \alpha_{n+1}-\alpha_n \vert d(Tx_n,u), \\[1mm]
d(x_{n+2}, x_{n+1}) & = d((1-\beta_{n+1})Uy_{n+1} + \beta_{n+1} y_{n+1}, (1-\beta_n)Uy_n + \beta_n y_n)\\
& \leq d((1-\beta_{n+1})Uy_{n+1} + \beta_{n+1} y_{n+1}, (1-\beta_{n+1})Uy_n + \beta_{n+1} y_n)  \\
& \quad + d((1-\beta_{n+1})Uy_n + \beta_{n+1} y_n, (1-\beta_n)Uy_n + \beta_n y_n)\\
&  \leq (1-\beta_{n+1}) d(Uy_{n+1},Uy_n)+\beta_{n+1} d(y_{n+1},y_n)  \\
& \quad + \vert \beta_{n+1} - \beta_n \vert d(Uy_n,y_n) \quad \text{by (W4) and (W2)}\\
&  \leq d(y_{n+1},y_n) + \vert \beta_{n+1} - \beta_n \vert d(Uy_n,y_n). \qedhere
\end{align*}
\end{proof}

We assume, furthermore, that  $T$ and $U$ have common fixed points and let $p$ be a common fixed point of $T$ and $U$. 
Define 
\begin{equation}
M_p= \max\{d(x_0, p), d(u,p)\}. \label{def-Mp}
\end{equation}

\blem\label{lemma-Mp} 
For all  $n\in \N$,
\begin{enumerate}
\item\label{lemma-Mp-i} $d(x_n, p)\leq M_p$ and $d(y_n, p)\leq M_p$.
\item\label{lemma-Mp-ii} $d(x_{n+1}, x_n)\leq 2M_p$ and $d(y_{n+1}, y_n)\leq 2M_p$.
\item\label{lemma-Mp-iii} $d(Tx_n,u) \leq 2M_p$ and $d(Uy_n,y_n)\leq 2M_p$.
\end{enumerate}
\elem
\begin{proof}
For \eqref{lemma-Mp-i} see \cite[Lemma 3.2]{DinPin21}.  \eqref{lemma-Mp-ii}, \eqref{lemma-Mp-iii} follow easily from \eqref{lemma-Mp-i}.
\end{proof}

\blem
The following hold for all  $n\in \N$:
\begin{align}
d(y_{n+1}, y_n) & \le (1-\alpha_{n+1}) d(x_{n+1}, x_n) + 2M_p\vert \alpha_{n+1}-\alpha_n \vert, \label{as-reg-dy}\\
\!\!\!\! d(x_{n+2}, x_{n+1}) & \le  (1-\alpha_{n+1}) d(x_{n+1}, x_n) + 2M_p(\vert \alpha_{n+1}-\alpha_n \vert + 
\vert \beta_{n+1} - \beta_n \vert). \label{as-reg-dx}
\end{align}
\elem
\begin{proof}
Apply \eqref{dynp1-yn},  \eqref{dxnp1-xn}, and Lemma~\ref{lemma-Mp}.\eqref{lemma-Mp-iii}. 
\end{proof}

\subsection{Quantitative lemmas}

The following lemma due to Xu \cite{Xu02} is well-known, as it is  one of the main tools in 
different convergence proofs for nonlinear iterations.

\blem\label{lemma-Xu02}
Let $(s_n)_{n\in\N}$ be a sequence of nonnegative real numbers satisfying, for all $n\in\N$,
\[
s_{n+1}\leq (1-a_n)s_n + a_n b_n+c_n,
\]
where $(a_n)_{n\in\N}$ is a sequence in $[0,1]$ such that  $\sum\limits_{n=0}^\infty a_n$ diverges,
$(b_n)_{n\in\N}$ is a sequence of real numbers with $\limsup\limits_{n\to\infty} b_n \leq 0$,
 and $(c_n)_{n\in\N}$ is a sequence of nonnegative reals such that $\sum\limits_{n=0}^\infty c_n$ 
 converges.

Then $\lim\limits_{n\to\infty} s_n =0$.
\elem

Quantitative versions of this lemma were given for the first time in \cite{KohLeu12a} for 
the particular case $c_n=0$ and for the general case in \cite[Section~3]{LeuPin21}. We refer also to \cite[Section~2.3]{DinPin21}
for an extensive discussion on different quantitative variants of Xu's lemma. 
We shall use in this paper a quantitative version of another particular case of Lemma~\ref{lemma-Xu02}, 
obtained by taking $b_n=0$. 

\begin{proposition}\label{qXu}
Let $(s_n), (c_n)\subseteq [0,+\infty)$  and $(a_n)\subseteq [0,1]$ satisfy, for all $n\in\N$,
\[
s_{n+1}\leq (1-a_n)s_n + c_n.
\]
Assume that $L\in\N^*$ is an upper bound on $(s_n)$,  $\sum\limits_{n=0}^\infty a_n$  diverges with rate of 
divergence $\theta$, and  $\sum\limits_{n=0}^\infty c_n$ converges with Cauchy modulus $\chi$.

Then $\lim\limits_{n\to\infty} s_n=0$ with rate of convergence $\Sigma$ defined by 
\begin{equation}
\Sigma(k)=\theta\big(\chi(2k+1)+1+\lceil \ln(2L(k+1))\rceil\big)+1.
\end{equation}
\end{proposition}
\begin{proof}
See \cite[Lemma 2.9(1)]{DinPin21}.
\end{proof}

Sabach and Shtern \cite[Lemma~3]{SabSht17} proved a particularly interesting version of Xu's 
lemma that allowed them, for a clever choice for the sequence $(a_n)$, to  obtain linear 
rates of asymptotic regularity for the sequential averaging method (SAM), a generalization 
of the Halpern iteration.  Recently, Cheval, Kohlenbach, and the first author \cite{CheKohLeu23} applied
\cite[Lemma~3]{SabSht17} to compute  linear rates for the Tikhonov-Mann iteration and
the modified Halpern iteration in $W$-hyperbolic spaces. 

In this paper we shall use a version of Sabach and Shtern's  lemma to obtain 
linear rates of asymptotic regularity for the alternating Halpern-Mann iteration, too. We give the proof, for completeness.

\begin{lemma}\label{lemma-SabSht-version}
Let $L>0$,  $J\geq N\geq 2$,  and $\gamma\in(0,1]$. Assume that $a_n=\frac{N}{\gamma(n+J)}$  and $c_n\leq L$ for all $n\in\N$.
Consider a sequence of nonnegative real numbers $(s_n)$ satisfying the following:
$s_0 \leq L$ and,  for all $n\in\N$,
\[
s_{n+1} \leq (1 - \gamma a_{n+1})s_n + (a_n-a_{n+1})c_n.
\]
Then 
\[
s_n\leq \frac{JL}{\gamma(n+J)} \quad \text{for all~} n\in \N.
\]
\end{lemma}
\begin{proof}
We show the result by induction on $n$. 

$n=0$:  We trivially have that $s_0 \leq \frac{L}{\gamma}$, as $\gamma \in(0,1]$. 

$n\Rightarrow n+1$: We get that
\begin{align*}
s_{n+1} &\leq (1- \gamma a_{n+1})s_n+(a_n-a_{n+1})L \\
& \leq \left(1-\frac{N}{n+1+J}\right)\frac{JL}{\gamma(n+J)} + \left(\frac{N}{\gamma(n+J)} - \frac{N}{\gamma(n+1+J)}\right)L \\
& \quad \text{by the induction hypothesis}\\
& = \frac{(n+1+J-N)JL}{\gamma(n+1+J)(n+J)} +\frac{NL}{\gamma(n+J)(n+1+J)}\\
& \leq \frac{(n+1+J-N)JL}{\gamma(n+1+J)(n+J)} +\frac{JL}{\gamma(n+J)(n+1+J)} \quad \text{as~}  J\geq N\\
&=\frac{(n+J+2-N)JL}{\gamma(n+J)(n+1+J)} \leq \frac{JL}{\gamma(n+1+J)} \quad \text{as~}  N\geq 2. \qedhere
\end{align*}
\end{proof}

\section{Main results}\label{main}

The main results of the paper provide effective rates of ($T$- and $U$-)asymptotic regularity of the alternating Halpern-Mann iteration.

Consider throughout that $X$ is a $W$-hyperbolic space, $C\subseteq X$  is a convex subset, and 
$T, U:C\to C$ are nonexpansive mappings. We assume that $T$ and $U$ have common fixed points, 
hence the set $Fix(T)\cap Fix(U)$ is nonempty, where $Fix(T)$ (resp. $Fix(U)$ is the set of fixed points of $T$ (resp. $U$). 

The sequences $(x_n)$, $(y_n)$ are defined by \eqref{HM},  $p\in Fix(T)\cap Fix(U)$, 
$M_p$ is given by \eqref{def-Mp}, and $K\in\N^*$  is such that $K\geq M_p$.

We shall use the following quantitative hypotheses on the parameter sequences $(\alpha_n)$, $(\beta_n)$
from \eqref{HM}:\\[2mm]
\begin{tabular}{lll}
$(Q1)$ & $\lim\limits_{n \to \infty} \alpha_n=0$ with rate of convergence $\roneq$; \\[1mm]
$(Q2)$ & $\sum\limits_{n=0}^\infty \alpha_n$ diverges with rate of divergence $\rtwoq$; \\[1mm]
$(Q3)$ & $\sum\limits_{n=0}^\infty  \vert \alpha_{n+1}-\alpha_n\vert$ converges with Cauchy modulus $\rthreeq$; \\[1mm]
$(Q4)$ & $\sum\limits_{n=0}^\infty  \vert \beta_{n+1}-\beta_n \vert$ converges with Cauchy modulus $\rfourq$; \\[1mm]
$(Q5)$ & $\rfiveq \in \N$ satisfies $\rfiveq \geq 2$ and $\frac1\rfiveq \leq \beta_n \leq 1- \frac1\rfiveq$ 
for all $n\in N$. \\[2mm]
\end{tabular}

These hypotheses were used in \cite{DinPin21} for the quantitative study of the Halpern-Mann iteration in CAT(0) spaces. 
\mbox{}

\subsection{Asymptotic regularity in $W$-hyperbolic spaces}

The first asymptotic regularity results are essentially reformulations of 
\cite[Pro-position 3.3(i),(ii)]{DinPin21} 
obtained by using $\frac1{k+1}$ instead of $\eps$. 

\begin{proposition}\label{As-reg-W-prop}
Assume that (Q2), (Q3), and (Q4) hold. Define 
\begin{equation}\label{def-main-chi}
\chi:\N\to\N, \quad  \chi(k)=\max\{\rthreeq(4K(k+1)-1), \rfourq(4K(k+1)-1)\}. 
\end{equation}
The following hold:
\begin{enumerate}
\item\label{xn-as-reg} $(x_n)$ is asymptotically regular with rate $\rarxn$ defined by 
\begin{equation}
\rarxn(k) = \rtwoq\left(\chi(2k+1)+2+\lceil \ln(4K(k+1))\rceil\right)+1. \label{def-rarxn}
\end{equation}
\item\label{yn-as-reg} $(y_n)$ is asymptotically regular with rate $\raryn$ defined by
\begin{equation}
\raryn(k)=\max \left\{\rarxn(2k+1), \rthreeq(4K(k+1)-1)+1\right\}.\label{def-raryn}
\end{equation}
\end{enumerate}
\end{proposition}
\begin{proof}
\begin{enumerate}
\item  As \eqref{as-reg-dx} holds, we can apply Proposition~\ref{qXu} with $s_n=d(x_{n+1}, x_n)$, $a_n=\alpha_{n+1}$, and
$c_n=2K(\vert \alpha_{n+1}-\alpha_n \vert +  \vert \beta_{n+1} - \beta_n \vert)$. 
By Lemma \ref{lemma-Mp}.\eqref{lemma-Mp-ii}, $L=2K$ is an upper bound for 
$(s_n)$. Furthermore, one can easily see that $\chi$ defined by \eqref{def-main-chi} is a Cauchy modulus for 
$\sum\limits_{n=0}^\infty c_n$ and that $\theta(n)=\rtwoq(n+1)$ is a rate of divergence for $\sum\limits_{n=0}^\infty a_n$.
\item  By \eqref{as-reg-dy}, we get that for all $n\in\N$, 
\begin{align*}
d(y_{n+1}, y_n) & \le  d(x_{n+1}, x_n) + 2K\vert \alpha_{n+1}-\alpha_n \vert. 
\end{align*}
The fact that $\raryn$ is a rate of asymptotic regularity of $(y_n)$ follows easily 
from \eqref{xn-as-reg} and (Q3). 
\end{enumerate}
\end{proof}

\begin{remark}
The rates $\rarxn$ and $\raryn$  depend on $X$, $C$, $T$, $U$, $x_0$, $u$ only via  $K$, 
an integer upper bound of $M_p$ defined by \eqref{def-Mp}. In particular, if $C$ is bounded, 
then one can take $M_p$ to be the diameter $d_C$ of $C$ and $K=\left\lceil d_C\right\rceil$. 
\end{remark}

\begin{remark}
Without loss of generality we can assume that $\rthreeq$ is increasing (by taking instead of $\rthreeq$ the increasing mapping
$\rthreeq^M(k)=\max\{\rthreeq(i)\mid i \leq k\}$). In this case, 
\begin{equation}\label{yn-as-reg-increasing}
\raryn(k)=\rarxn(2k+1).
\end{equation}
\end{remark}
\begin{proof}
For all $k\in\N$, we have that $\rtwoq(k+1)\geq k$, as $\alpha_k\leq 1$. It follows that 
\begin{align*}
\rarxn(2k+1) & > \chi(4k+3) +2 \geq  \rthreeq(16K(k+1)-1) + 2 \\
& > \rthreeq(4K(k+1)-1)+1. \qedhere
\end{align*}
\end{proof}

By Remark~\ref{remark-HM-TM-mH}\eqref{TM-mH}, for $T=id_C$, $\rarxn$  (resp. $\raryn$) is a rate of asymptotic regularity of the Tikhonov-Mann (resp. modified Halpern) iteration. 
One can easily see that $\rarxn$ is slightly better than the rate obtained in \cite[Theorem~4.1(i)]{CheLeu22}. 
Furthermore, in the case that $\rthreeq$ is increasing, $\raryn$ given by \eqref{yn-as-reg-increasing} has a similar 
form with the rate computed in \cite[Proposition~4.4(i)]{CheKohLeu23}.

\subsubsection{Linear rates}\label{s-linear-rates}

Let us consider the following parameter sequences:
\begin{center}
$\alpha_n=\frac{2}{n+2}$ and $\beta_n=\beta\in(0,1)$.
\end{center}
As pointed out in \cite{DinPin21}, one can apply Proposition~\ref{As-reg-W-prop} to get exponential rates of asymptotic regularity.
The next result shows that we can  obtain, as an application of Lemma~\ref{lemma-SabSht-version}, linear rates of asymptotic regularity.
 
\begin{proposition}\label{linear-rate-as-reg}
For all $n\in \N$, 
\begin{equation}
d(x_n,x_{n+1})\leq \frac{4K}{n+2} \quad \text{and} \quad d(y_n,y_{n+1})\leq \frac{4K}{n+3}.
\end{equation}
Thus, $(x_n)$ is asymptotically regular with rate $\rarxns(k) = 4K(k+1)-2$ and  $(y_n)$ is asymptotically regular with rate 
$\raryns(k) = 4K(k+1)-3$. 
\end{proposition}
\begin{proof}
Applying \eqref{as-reg-dx}, we get that for all $n\in\N$,
\begin{align}
d(x_{n+2}, x_{n+1}) & \le  \left(1-\frac{2}{n+3}\right) d(x_n,x_{n+1}) + \left(\frac{2}{n+2}-\frac{2}{n+3}\right)2K.
\end{align}
One can easily see that we can apply Lemma~\ref{lemma-SabSht-version} with $s_n=d(x_n,x_{n+1})$, 
$L=2K$, $N=J=2$, $\gamma=1$, $a_n=\alpha_n=\frac2{n+2}$, and
$c_n=2K$ to get that for all $n\in\N$, 
\[d(x_n,x_{n+1})\leq \frac{4K}{n+2}.\]
Furthermore, by \eqref{as-reg-dy}, we have  that for all $n\in\N$, 
\[d(y_n,y_{n+1})\leq  \frac{4K}{n+3}.\]
The conclusion follows immediately. 
\end{proof}

\subsection{$T$- and $U$-asymptotic regularity in $UCW$-hyperbolic spaces}

Rates of $T$- and $U$-asymptotic regularity for $(x_n)$, $(y_n)$ were also computed in \cite{DinPin21} in the setting of CAT(0) spaces. 
We now show that these results can be generalized to $UCW$-hyperbolic spaces.
 
\bprop\label{Delta-As-reg-UCW}
Let $(X,\eta)$ be a $UCW$-hyperbolic space. Assume that (Q1) and (Q5)  hold and that 
$(x_n)$ is asymptotically regular with rate $\Delta$. 
 
The following are satisfied:
\begin{enumerate}
 \item\label{U-As-reg-yn-UCW} $(y_n)$ is $U$-asymptotically regular with rate $\uraryn$ defined by 
\[\uraryn(k)=\max\{ \Delta(2P(k+1)-1), \roneq(2PK(k+1)-1)\},\]
where $\displaystyle P=\left\lceil\frac{\rfiveq^2}{\eta\left(K, \frac1{K(k+1)}\right)}\right\rceil$.
\item $\lim\limits_{n\to\infty} d(x_n,y_n)=0$ with rate of convergence
\begin{equation}
\rxnyn(k)=\max\{\Delta(2k+1), \uraryn(2k+1)\}. 
\end{equation}
\item  $(y_n)$ is $T$-asymptotically regular with rate  
\begin{equation}
\traryn(k)=\max\{\rxnyn(2k+1), \roneq(4K(k+1)-1)\}. 
\end{equation}
\item $(x_n)$ is $U$-asymptotically regular with rate  
\begin{equation}
\urarxn(k)=\max\{\rxnyn(4k+3),  \uraryn(2k+1) \}.
\end{equation}
\item $(x_n)$ is $T$-asymptotically regular with rate  
\begin{equation}
\trarxn(k)=\max\{\rxnyn(4k+3), \traryn(2k+1)\}.
\end{equation}
\end{enumerate}
\eprop
\begin{proof}
\begin{enumerate}
\item Let $k\in\N$ and $n\geq \uraryn(k)$. Assume that $d(Uy_n,y_n)>\frac1{k+1}$.
Since $d(y_n,p)\le K$, by Lemma~\ref{lemma-Mp}\eqref{lemma-Mp-i}, and $d(Uy_n,y_n)\leq d(Uy_n,p)+d(y_n,p) \leq 2d(y_n,p)$, we get that 
\begin{equation}
\frac1{2(k+1)} < d(y_n,p) \leq K. \label{2dynp}
\end{equation}
Furthermore, $d(Uy_n,p) \leq d(y_n,p)\leq K$, $d(Uy_n,y_n) > \frac1{K(k+1)}d(y_n,p)$, and $\frac1{K(k+1)} <  2$.
It follows that we can apply 
Lemma~\ref{prop_eta-s-r}  with $x=Uy_n$, $y=y_n$, $a=p$,  $r=d(y_n,p)$, $\eps=\frac1{K(k+1)}$,
 $\lambda=\beta_n$, and $s=K$ to conclude that
\begin{align*}
	d(x_{n+1},p)  & = d((1-\beta_n)Uy_n + \beta_ny_n, p) \\
	& \leq \left( 1- 2\beta_n(1-\beta_n)\eta\left(K, \frac1{K(k+1)}\right) \right)d(y_n,p)\\
	& = d(y_n,p)- 2d(y_n,p)\beta_n(1-\beta_n)\eta\left(K, \frac1{K(k+1)}\right)\\
	& \leq d(y_n,p)- 2d(y_n,p) \frac1{\rfiveq^2}\eta\left(K, \frac1{K(k+1)}\right)\\
    & \quad \text{as, ~by (Q5),~}  \beta_n, 1-\beta_n\geq \frac1{\rfiveq}\\
	& < d(y_n,p)- \frac1{(k+1)\rfiveq^2}\eta\left(K, \frac1{K(k+1)}\right) \quad \text{by \eqref{2dynp}}.
\end{align*}

Since, by (W1), $d(y_n,p)\leq (1-\alpha_n)d(Tx_n,p)+\alpha_nd(u,p) \leq d(x_n,p)+\alpha_nK$, we get that 
\begin{align*}
d(x_{n+1},p) &  < d(x_n,p) + \alpha_nK - \frac1{(k+1)\rfiveq^2}\eta\left(K, \frac1{K(k+1)}\right).
\end{align*}
It follows that 
\begin{align*}
\frac{1}{P(k+1)} &  \leq \frac1{(k+1)\rfiveq^2}\eta\left(K, \frac1{K(k+1)}\right)\\
& < d(x_n,p)-d(x_{n+1},p) + \alpha_nK \leq d(x_{n+1},x_n) + \alpha_nK \\
	& \leq  \frac{1}{P(k+1)}, \text{~as~}n\geq \uraryn(k).
\end{align*}
We have obtained a contradiction. 
\item[(ii)-(v)] are obtained easily from the following inequalities
\begin{align*}
d(x_n,y_n) & \leq d(x_{n+1},x_n) + d(x_{n+1},y_n) \\
& \stackrel{(W7)}{=} d(x_{n+1},x_n) + (1-\beta_n)d(Uy_n,y_n)\\
& \leq d(x_{n+1},x_n) + d(Uy_n,y_n),\\
d(Ty_n,y_n) & \leq d(Ty_n,Tx_n) + d(Tx_n, y_n) \stackrel{(W7)}{=} d(Tx_n,Ty_n) + \alpha_nd(Tx_n,u)\\
		& \leq d(x_n,y_n) + 2K\alpha_n,\\
d(Ux_n,x_n) & \leq d(Ux_n,Uy_n)+d(Uy_n, y_n)+d(y_n,x_n)\\
& \leq 2d(x_n,y_n) + d(Uy_n,y_n),\\
d(Tx_n,x_n) & \leq d(Tx_n,Ty_n)+d(Ty_n, y_n)+d(y_n,x_n)\\
& \leq 2d(x_n,y_n) + d(Ty_n,y_n). \qedhere
\end{align*}
\end{enumerate}
\end{proof}

\begin{corollary}\label{as-reg-UCW}
Assume  that $(X,\eta)$ is a $UCW$-hyperbolic space and (Q1)-(Q6) hold. Then $(x_n)$ and $(y_n)$ are 
$U$- and $T$-asymptotically regular with rates obtained from the ones in Proposition~\ref{Delta-As-reg-UCW}
by replacing $\Delta$ with $\rarxn$.
\end{corollary}
\begin{proof}
By Proposition~\ref{As-reg-W-prop}\eqref{xn-as-reg},  under the hypotheses
(Q2), (Q3), (Q4),  $(x_n)$ is asymptotically regular with rate $\rarxn$. 
\end{proof}

The following observation is inspired by \cite[Remark 15]{Leu07}; see also \cite[Theorem~3.4]{Koh03} for 
a similar remark in the context of uniformly convex normed spaces.

\begin{remark}\label{eta-simple}
Assume that 
\begin{center}
(*) \quad $\eta(r,\eps)= \eps\cdot \tilde{\eta}(r,\eps)$ for some $\tilde{\eta}$ that 
increases with $\eps$ (for a fixed $r$). 
\end{center}
Then  Proposition~\ref {Delta-As-reg-UCW}\eqref{U-As-reg-yn-UCW} holds with  
\[\uraryntilde(k)=\max\{ \Delta(2\tilde{P}(k+1)-1), \roneq(2\tilde{P}K(k+1)-1)\}.\]
where $\tilde{P}=\left\lceil\frac{\rfiveq^2}{\tilde{\eta}\left(K, \frac1{K(k+1)}\right)}\right\rceil$.
\end{remark}
\begin{proof}
Let $k\in\N$ and $n\geq \uraryntilde(k)$. Assume that $d(Uy_n,y_n)>\frac1{k+1}$. Follow the proof of 
Proposition~\ref{Delta-As-reg-UCW}.\eqref{U-As-reg-yn-UCW}, but replace $\eps=\frac1{K(k+1)}$ 
with $\eps=\frac1{d(y_n,p)(k+1)}$. Then we get that 
\begin{align*}
d(x_{n+1},p)  & \leq d(y_n,p)- \frac2{(k+1)\rfiveq^2}\tilde{\eta}\left(K, \frac1{d(y_n,p)(k+1)}\right)\\
	& < d(y_n,p)- \frac1{(k+1)\rfiveq^2}\tilde{\eta}\left(K, \frac1{d(y_n,p)(k+1)}\right).
\end{align*}
As $d(y_n,p)\le K$ and $\tilde{\eta}$ is increasing with $\eps$, we have that $\tilde{\eta}\left(K, \frac1{d(y_n,p)(k+1)}\right)\geq 
\tilde{\eta}\left(K, \frac1{K(k+1)}\right)$.
It follows that 
\begin{align*}
	d(x_{n+1},p) < d(y_n,p)- \frac1{(k+1)\rfiveq^2}\tilde{\eta}\left(K, \frac1{K(k+1)}\right).
\end{align*}
Continue as in the proof of Proposition~\ref{Delta-As-reg-UCW}\eqref{U-As-reg-yn-UCW} with $\tilde{P}, \tilde{\eta}$ instead of $P,\eta$.
\end{proof}

\subsubsection{Quadratic rates in CAT(0) spaces}\label{quadratic-CAT0}

For the remainder of this section, consider $X$ to be a CAT(0) space. As CAT(0) spaces are $UCW$-hyperbolic spaces with modulus 
$\eta(r,\eps)=\frac{\eps^2}8$ that obviously satisfies (*) from Remark~\ref{eta-simple} with 
$\tilde{\eta}(r,\eps)=\frac{\eps}8$, we get the following. 

\begin{proposition}\label{U-as-reg-yn-CAT0}
Assume that (Q1) and (Q5)  hold and that 
$(x_n)$ is asymptotically regular with rate $\Delta$.  

Then  $(y_n)$ is $U$-asymptotically regular with rate $\uraryncat$ defined by 
\[\uraryncat(k)=\max\{ \Delta(2\pcat(k+1)-1), \roneq(2\pcat K(k+1)-1)\},\]
where $\pcat=8K\rfiveq^2(k+1)$.
\end{proposition}

Furthermore, Proposition~\ref{Delta-As-reg-UCW} holds with $\uraryn$ replaced by $\uraryncat$. 
As in Corollary~\ref{as-reg-UCW}, if, moreover, (Q2)-(Q4) hold, then one can take $\Delta=\rarxn$, 
with $\rarxn$ defined by \eqref{def-rarxn}. One can easily verify that our results, when particularized to CAT(0) spaces, actually recover the rates of ($T$- and $U$-)asymptotic regularity from \cite{DinPin21}.

For the particular parameter sequences considered in Subsubsection~\ref{s-linear-rates}, 
we get quadratic rates of $T$- and $U$-asymptotic regularity.

\begin{proposition}
Let $\alpha_n=\frac{2}{n+2}$ and $\beta_n=\beta\in(0,1)$.  Define 
\begin{equation} \label{def-Lambda-beta}
\rfiveq =\left\lceil\max\left\{\frac1\beta, \frac1{1-\beta}\right\}\right\rceil.
\end{equation}
The following hold:
\begin{enumerate}
\item $\uraryns(k)\!=\!2^6K^2 \rfiveq^2(k+1)^2\!-\!2$ is a rate of $U$-asymptotic regularity of $(y_n)$.
\item $\lim\limits_{n\to\infty} d(x_n,y_n)=0$ with rate of convergence $\rxnyns(k)= 2^8K^2 \rfiveq^2(k+1)^2\!-\!2$.
\item  $\traryns(k)=2^{10}K^2 \rfiveq^2(k+1)^2\!-\!2$ is a rate of $T$-asymptotic regularity of $(y_n)$.
\item $\urarxns(k)=2^{12}K^2 \rfiveq^2(k+1)^2\!-\!2$ is a rate of $T$- and $U$-asymptotic regularity of $(x_n)$. 
\end{enumerate}
\end{proposition}
\begin{proof}
Obviously, $\roneq(k)=2k$ is a rate of convergence for $\left(\frac{2}{n+2}\right)$.  By 
Proposition~\ref{linear-rate-as-reg}, $\rarxns(k)=4K(k+1)-2$ is a rate of asymptotic regularity of  $(x_n)$.  Furthermore, $(Q5)$ holds with $\Lambda$ defined by \eqref{def-Lambda-beta}.
To obtain (i) apply Proposition~\ref{U-as-reg-yn-CAT0} with $\Delta=\rarxns$.
For (ii)-(iv) we use  Proposition~\ref{Delta-As-reg-UCW}(ii)-(v) with $\Delta=\rarxns$ and $\uraryn=\uraryns$. 
\end{proof}

\mbox{}

\section*{Acknowledgements}

The first author thanks Adriana Nicolae for useful discussions on the subject of the paper.
The second author was supported by the German Science Foundation (DFG Project KO 1737/6-2).

\end{document}